\documentclass[11pt,a4paper]{amsart}
\usepackage{amssymb,xspace}
\usepackage{amstext}
\theoremstyle{plain}
\usepackage{amsbsy,amssymb,amsfonts,latexsym}
\usepackage[utf8]{inputenc}
\usepackage{xcolor}

\usepackage{enumerate}
\usepackage{graphics}

\date{\today}
\title{Hypercyclic algebras}
\author{Fr\'ed\'eric Bayart}
\address{Université Clermont Auvergne, CNRS, LMBP, F-63000 Clermont–Ferrand, France.}
\email{Frederic.Bayart@uca.fr}
\thanks{The author was partially supported by the grant ANR-17-CE40-0021 of the French National Research Agency ANR (project Front)}



\newcommand{\veps}{\varepsilon}

\def\RR{\mathbb R}

\def\NN{\mathbb N}

\def\TT{\mathbb T}
\def\DD{\mathbb D}

\def\CC{\mathbb C}

\def\card{\textrm{card}}

\def\bl{\mathbf k}
\def\bj{\mathbf j}
\def\bl{\mathbf l}
\def\NNinf{\NN^{(\infty)}}
\def\NNinfe{\NN^{(\infty)}\backslash\{(0,\dots)\}}

\newtheorem{theorem}{Theorem}[section]

\newtheorem{lemma}[theorem]{Lemma}

\newtheorem{proposition}[theorem]{Proposition}

\newtheorem{corollary}[theorem]{Corollary}

{\theoremstyle{definition}}
{\theoremstyle{definition}}

{\theoremstyle{definition}\newtheorem{example}[theorem]{Example}}

{\theoremstyle{definition}}

{\theoremstyle{definition}}

{\theoremstyle{definition}\newtheorem{remark}[theorem]{Remark}}

\begin{document}

\begin{abstract}
We prove the existence of algebras of hypercyclic vectors in three cases: convolution operators, composition operators, and backward shift operators.
\end{abstract}

\maketitle

\section{Introduction}
When we work with a linear dynamical system $(X,T)$, it is natural to study how its linear properties  influence its dynamical properties. Here, $X$ denotes a topological vector space and $T$
is a continuous linear operator on $X$. We are mainly interested in hypercyclic operators: there exists $x\in X$, called a hypercyclic vector for $T$, such that 
$\{T^n x;\ n\geq 0\}$ is dense in $X$ (when $X$ is a second-countable Baire space, this amounts to saying that $T$ is topologically transitive). 
It is well known that the linear properties of $(X,T)$ reflects on $HC(T)$, the set of hypercyclic vectors for $T$: $HC(T)\cup\{0\}$ always contains a dense subspace (\cite{Bou93}) and there are nice criteria to determine if it contains a closed infinite-dimensional subspace (see \cite{GoLM00}, \cite{Men14}).

In this paper we assume that $X$ is also a topological algebra and we ask whether $HC(T)\cup\{0\}$ contains a nontrivial algebra; we will call this a hypercyclic algebra for $T$. 
We explore three relevant situations.

\subsection{Convolution operators}

Following the pioneering work of Birkhoff and MacLane, Godefroy and Shapiro have shown in \cite{GoSh91}
that a nonconstant operator which commutes with all translations $\tau_a$ is hypercyclic on $H(\CC)$.
Such an operator may be written $\phi(D)$, where $\phi$ is an entire function of exponential type and $D$ is the derivation operator.
Bayart and Matheron in \cite{BM09} and independently Shkarin in \cite{Shk10b} have shown that $D$ admits a hypercyclic algebra.
The argument of \cite{BM09}, which is based on the Baire category theorem and the fact that $\bigcup_n \ker(D^n)$ is dense in $H(\CC)$,
was refined by B\`es, Conejero and Papathanasiou in \cite{BCP17} to prove that $P(D)$ supports a hypercyclic algebra for all nonzero polynomials $P$ with $P(0)=0$
(see also the recent paper \cite{FalGre18} for the existence of hypercyclic algebras for weighted backward shifts in various Fr\'echet algebras).

More recently, in the nice paper \cite{BCP18}, the same authors provide further examples of entire functions $\phi$ such that $\phi(D)$ admits such an algebra. For instance, this holds true for $\phi(z)=\cos(z)$ which does not satisfy $\phi(0)=0$ and which is not a polynomial.

In stark contrast with this, it was observed in \cite{ACPS07} that the orbit of $f^2$ under $\lambda\tau_a$ (which corresponds to the case $\phi(z)=\lambda e^{az}$)
can only contain functions for which the multiplicities of their zeros is even. 

Our first main theorem characterizes the existence of a hypercyclic algebra for $\phi(D)$ when $|\phi(0)|<1$, or when $|\phi(0)|=1$ and $\phi$ has moderate growth.

\begin{theorem}\label{thm:mainconvolution}
Let $\phi$ be a nonconstant entire function with exponential type. 
\begin{enumerate}
\item Assume that $|\phi(0)|<1$. Then the following assertions are equivalent:
\begin{enumerate}[(i)]
\item $\phi(D)$ supports a hypercyclic algebra.
\item $\phi$ is not a multiple of an exponential function.
\end{enumerate}
\item Assume that $|\phi(0)|=1$ and $\phi$ has subexponential growth. If either $\phi'(0)\neq 0$ or $\phi$ has order less than 1/2, then $\phi(D)$ supports a hypercyclic algebra.
\end{enumerate}
\end{theorem}

In view of the previous result, it is tempting to conjecture that the assumption $|\phi(0)|\leq 1$ is a necessary condition for $\phi(D)$ to admit a hypercyclic algebra.
Surprizingly, this is not the case.

\begin{theorem}\label{thm:exampleconvolution}
Let $\phi(z)=2\exp(-z)+\sin(z)$. Then $\phi(D)$ supports a hypercyclic algebra.
\end{theorem}

Another natural conjecture is that $\phi(D)$ always suppport a hypercyclic algebra as soon as $\phi$ is not a multiple of an exponential function.
We do not know if this conjecture is true. Nevertheless, we still get an interesting result if we weaken the conclusion.
Recall that a vector $x\in X$ is a supercyclic vector for $T\in\mathcal L(X)$ provided $\{\lambda T^n x;\ \lambda\in\CC,\ n\geq 0\}$ is dense in $X$. 
If $X$ is a topological algebra, then any subalgebra of $X$ consisting entirely (but zero) of supercyclic vectors for $T$ is said to be a supercyclic algebra.

\begin{theorem}\label{thm:mainpowers}
Let $\phi$ be a nonconstant entire function with exponential type. The following assumptions are equivalent:
\begin{enumerate}[(i)]
\item $\phi(D)$ supports a supercyclic algebra.
\item There exists $f\in H(\CC)$ such that, for all $m\geq 1$, $f^m$ is a hypercyclic vector for $\phi(D)$.
\item $\phi$ is not a multiple of an exponential function.
\end{enumerate}
\end{theorem}

To our knowledge, the existence of $f\in H(\CC)$ such that $f^m\in HC\big(\phi(D)\big)$ for all $m\geq 1$
was only known when $\phi$ has subexponential type \cite{Be97}.

\subsection{Composition operators}

Let $\Omega\subset\CC$ be a simply connected domain and let $\varphi$ be a holomorphic self-map of $\Omega$. The composition operator $C_\varphi(f)=f\circ\varphi$ is a bounded operator on $H(\Omega)$ which is hypercyclic if and only if $\varphi$ is univalent and has no fixed point in $\Omega$. Moreover, in that case, if $P$ is a nonconstant polynomial, then $P(C_\varphi)$ is hypercyclic (see \cite{Bes13}). It was observed in \cite{BCP18} that $C_\varphi$ never supports a hypercyclic algebra and it was asked whether $P(C_\varphi)$ can support such an algebra. We provide an affirmative answer.

\begin{theorem}\label{thm:maincomposition}
Let $\Omega\subset\CC$ be a simply connected domain and let $\varphi$ be a holomorphic self-map of $\Omega$ which is univalent and has no fixed point in $\Omega$. Let also $P$ be a nonconstant polynomial which is not a multiple of $z$ and which satisfies $|P(1)|<1$. Then $P(C_\varphi)$ supports a hypercyclic algebra.
\end{theorem}

\subsection{Backward shift operators}
So far, our examples live only in $F$-algebras (namely in metrizable and complete topological algebras without assuming that the distance is
induced by a norm) and the proofs of Theorems \ref{thm:mainconvolution} and 
\ref{thm:maincomposition} depend heavily on the non Banach structure of the ambient space.
We provide now an example in the Banach algebra $\ell^1(\NN)=\left\{(u_n)_{n\geq 0};\ \sum_n |u_n|<+\infty\right\}$ endowed with the convolution product $(u\star v)(k)=\sum_{j=0}^k u_j v_{k-j}.$ It was already sketched in \cite{BM09} that, denoting by $B$ the backward shift operator, $HC(2B)\cup\{0\}$ contains a nontrivial algebra. Again, the density of $\bigcup_n\ker (B^n)$ was the key for the proof.

We go much further (we denote by $\DD$ the open unit disk and by $\TT$ its boundary the unit circle).
\begin{theorem}\label{thm:mainbackward}
Let $P\in\CC[X]$ be nonconstant and let $B$ be the backward shift operator on $\ell^1(\NN)$. Then the following assertions are equivalent.
\begin{enumerate}[(i)]
\item $P(B)$ is hypercyclic.
\item $P(B)$ admits a hypercyclic algebra.
\item $P(\DD)\cap\TT\neq\varnothing$.
\end{enumerate}
\end{theorem}

\subsection{Organization of the paper and strategy for the proofs.}
The proofs of our results develop a method initiated in \cite{BCP18} and use eigenvalues and eigenvectors of our operators.
In Section \ref{sec:smalleigenvalues}, we use eigenvalues of modulus slightly bigger than 1 to prove half of Theorem \ref{thm:mainconvolution} and Theorem \ref{thm:maincomposition}. 
In Section \ref{sec:largeeigenvalues}, we use eigenvalues with large modulus to prove the remaining part of Theorem \ref{thm:mainconvolution} whereas in Section \ref{sec:unimodular}, we use unimodular eigenvalues to prove Theorem \ref{thm:mainbackward}. There is also a significant difference between the first two cases and the last one: the product of two eigenvectors can or cannot be still an eigenvector. 
The latter situation is of course more difficult!
In Section \ref{sec:powers}, we come back to convolution operators and study the case $|\phi(0)|>1$. As a consequence we prove Theorems \ref{thm:exampleconvolution} and \ref{thm:mainpowers}.

Of course, if we know that an operator admits a hypercyclic algebra, then it is natural to ask how big it can be. Can it be dense? Can it be infinitely generated? These questions were investigated
very recently in \cite{BP18} where it is shown that for most of the examples exhibited in \cite{BCP18}, one can improve the construction to get a dense and infinitely generated hypercyclic algebra.
In Section \ref{sec:infinitelygenerated}, we show how to modify our proofs to obtain a similar improvement. We choose to postpone this in a separate section because the arguments
become more technical and we think that the ideas appear more clearly by handling separately the case of singly generated algebras.

We end up this introduction with the following lemma taken from \cite[Remark 8.28]{BM09}, which gives a criterion for the existence of a hypercyclic algebra.
It can be seen as a strong form of the property of topological transitivity. 

\begin{lemma}\label{lem:criterion}
Let $T$ be a continuous operator on some separable $F$-algebra $X$. Assume that, for any pair $(U,V)$ of nonempty open sets in $X$, for any open neighbourhood $W$ of zero, 
and for any positive integer $m$, one can find $u\in U$ and an integer $N$ such that $T^N(u^n)\in W$ for all $n<m$ and $T^N(u^m)\in V$. Then $T$ admits a hypercyclic algebra.
\end{lemma}

%
%
%

\section{Small eigenvalues}\label{sec:smalleigenvalues}

\subsection{A general result for operators with small eigenvalues}

We shall deduce part of Theorem \ref{thm:mainconvolution} and Theorem \ref{thm:maincomposition} from the following general result.

\begin{theorem}\label{thm:philarge}
Let $X$ be an F-algebra and let $T\in\mathcal L(X)$. Assume that there exist a function $E:\CC\to X$ and
an entire function $\phi:\CC\to \CC$ satisfying the following assumptions:
\begin{enumerate}
\item for all $\lambda\in\CC$,\ $TE(\lambda)=\phi(\lambda)E(\lambda)$;
\item for all $\lambda,\mu\in\CC$, $E(\lambda)E(\mu)=E(\lambda+\mu)$;
\item for all $\Lambda\subset\CC$ with an accumulation point, $\textrm{span}\left(E(\lambda);\ \lambda\in\Lambda\right)$ is dense in $X$;
\item $\phi$ is not a multiple of an exponential function;
\item for all $m\in\NN$, there exist $a,b\in\CC$ such that $|\phi(mb)|>1$ and, for all $n\in\{1,\dots,m\}$, all $d\in\{0,\dots,n\}$, with $(n,d)\neq (m,m)$, $|\phi(db+(n-d)a)|<1$.
\end{enumerate}
Then $T$ supports a hypercyclic algebra.
\end{theorem}

We start with a lemma which explains why we have to exclude multiples of exponential functions.

\begin{lemma}\label{lem:convexity}
 Let $\phi$ be an entire function which is not a multiple of an exponential function. Then, for any $w_0\in\CC$
 with $\phi(w_0)\neq 0$ and any $\delta>0$, there exist $w_1,w_2\in B(w_0,\delta)$, $w_1\neq w_2$, such that the map
 $[0,1]\to\RR$, $t\mapsto \log|\phi(tw_1+(1-t)w_2)|$ is stricly convex.
\end{lemma}
\begin{proof}
 Since $\phi(w_0)\neq 0$, there exist some neighbourhood $V$ of $w_0$ and a holomorphic function $h:V\to\CC$
 such that $\phi(z)=\exp\big(h(z)\big)$ for all $z\in V$. Since $\phi$ is not a multiple of an exponential function,
 we know that $h$ is not an affine map. Thus there exists $w_1\in B(w_0,\delta)\cap V$ such that $h''(w_1)\neq 0$. Without loss
 of generality, we assume that $w_1=0$ and we write $h(z)=\sum_{k=0}^{+\infty}a_k z^k$. Then 
 \begin{align*}
  \log|\phi(z)|&=\Re e \big(h(z)\big)\\
  &=\Re e(a_0)+\left(\Re e(a_1)x-\Im m(a_1)y\right)\\
  &\quad+\left(\Re e(a_2)\left(x^2-y^2\right)-2\Im m(a_2)xy\right)+o\left(x^2+y^2\right)
 \end{align*}
if $z=x+iy$. Since $a_2\neq 0$, one may find $(x_0,y_0)\in\RR^2$ with 
$$\Re e(a_2)\left(x_0^2-y_0^2\right)-2 \Im m(a_2)x_0y_0>0.$$
Then $g(t)=\Re e \left(h\left(t\left(x_0+iy_0\right)\right)\right)=b_0+b_1t+b_2t^2+o(t^2)$ with $b_2>0$ is strictly convex around $0$.
 \end{proof}
 
 \begin{proof}[Proof of Theorem \ref{thm:philarge}]
  Let $U,V,W$ be nonempty open subsets of $X$ with $0\in W$ and let $m\geq 1$. By Lemma \ref{lem:criterion},
  it suffices to find $u\in U$ and $N\in\NN$ so that
\begin{align}
  T^N(u^n)&\in W,\ n=1,\dots,m-1\label{eq:belongs1}\\
  T^N(u^m)&\in V.\label{eq:belongs2}
\end{align}

The assumptions give us for this value of $m$ two complex numbers $a$ and $b$. We set $w_0=mb$. We then consider $\delta>0$ very small and $w_1,w_2\in B(w_0,\delta)$ such that 
\begin{itemize}
\item $t\in[0,1]\mapsto \log\left|\phi\left(tw_1+(1-t)w_2\right)\right|$ is strictly convex;
\item $|\phi|>1$ on $[w_1,w_2]$;
\item for all $n\in\{1,\dots,m\}$, for all $d\in\{0,\dots,n\}$ with $(n,d)\neq (m,m)$, 
for all $\lambda_1,\dots,\lambda_d\in [w_1,w_2]$ and all $\gamma_1,\dots,\gamma_{n-d}\in B(a,\delta)$, 
\begin{equation}\label{eq:philarge}
\left|\phi\left(\frac{\lambda_1+\cdots+\lambda_d}m+\gamma_1+\cdots+\gamma_{n-d}\right)\right|<1.
\end{equation}
\end{itemize}
We may ensure this last property because
$$\frac{\lambda_1+\cdots+\lambda_d}m+\gamma_1+\cdots+\gamma_{n-d}={db}+(n-d)a+z$$
where the norm of $z$ is controlled by $\delta$.

Since $B(a,\delta)$ and $[w_1,w_2]$ have accumulation points, we may find 
 $p,q\in\NN$, complex numbers $a_1,\dots,a_p$, $b_1,\dots,b_q$,
 complex numbers $\gamma_1,\dots,\gamma_p\in B(a,\delta)$ and complex numbers $\lambda_1,\dots,\lambda_q$ in $[w_1,w_2]$ such that
$$\sum_{l=1}^p a_l E\left(\gamma_l\right)\in U\textrm{ and }\sum_{j=1}^q b_j E\left(\lambda_j \right)\in V.$$
For $N\geq 1$ and $j\in\{1,\dots,q\}$, let $c_j:=c_j(N)$ be any complex number satisfying $c_j^m=b_j/\big(\phi(\lambda_j)\big)^{N}$ and
define
$$u:=u_N=\sum_{l=1}^p a_l E\left(\gamma_l\right)+\sum_{j=1}^q c_j E\left(\lambda_j/m\right).$$
We claim that, for $N$ large enough, $u_N$ belongs to $U$ and satisfies \eqref{eq:belongs1} and \eqref{eq:belongs2}.
That $u_N$ belongs to $U$ is clear, since $c_j(N)$ tends to zero as $N$ goes to infinity. In order to prove the other points,
we need to compute $u^n$ for $1\leq n\leq m$. To simplify the notations, let $I_p=\{1,\dots,p\}$ and for a multi-index
$\mathbf l\in I_p^d$, $a_{\mathbf l}$ will stand for $a_{l_1}\cdots a_{l_d}$ with the convention that an empty product is equal to 1. Then we may write
$$u^n=\sum_{d=0}^n \sum_{\substack{\bl\in I_p^{n-d} \\\bj\in I_q^d}}\alpha(\bl,\bj,d,n)a_\bl c_\bj E\left(\gamma_{l_1}+\cdots+\gamma_{l_{n-d}}+\frac{\lambda_{j_1}+\cdots+\lambda_{j_d}}m\right)$$
for some coefficients $\alpha(\bl,\bj,d,n)$ that we do not try to compute, but which does not depend on $N$.
To prove that $T^N (u^n)$ belongs to $W$ for $N$ large enough and $1\leq n<m$, we only have to prove that, for any $d\in\{0,\dots,n\}$ and any $\bl\in I_{p}^{n-d}$, any $\bj\in I_q^d$,
$$c_\bj(N)\left(\phi\left(\gamma_{l_1}+\cdots+\gamma_{l_{n-d}}+\frac{\lambda_{j_1}+\cdots+\lambda_{j_d}}m\right)\right)^N\longrightarrow_{N\to+\infty}0.$$
This follows from \eqref{eq:philarge} and the fact that $c_\bj(N)$ tends to zero.

The case $n=m$ is slightly different. We denote by $D_q$ the diagonal of $I_q^m$, namely 
the $m$-uples $(j,\dots,j)$, $1\leq j\leq q$. Then we decompose $u^m$ into
\begin{align*}
 u^m&=\sum_{d=0}^{m-1}\sum_{\substack{\bl\in I_p^{m-d} \\\bj\in I_q^d}}\alpha(\bl,\bj,d,m)a_\bl c_\bj E\left(\gamma_{l_1}+\cdots+\gamma_{l_{n-d}}+\frac{\lambda_{j_1}+\cdots+\lambda_{j_d}}m\right)\\
 &\quad+\sum_{\bj\in I_q^m\backslash D_q}\alpha(\bj,m)c_\bj E\left(\frac{\lambda_{j_1}+\cdots+\lambda_{j_m}}m\right)\\
 &\quad+\sum_{j=1}^q c_j^m E\left(\lambda_j\right)\\ 
 &=: v_1+v_2+v_3.
\end{align*}
The same considerations as above show that $T^N v_1$ tends to zero as $N$ goes to infinity. That $T^N v_2$ tends also to zero follows from a convexity argument. 
Indeed, for $\bj\in I_q^m\backslash D_q$, the strict convexity of the map $t\mapsto \log\left|\phi\left(tw_1+(1-t)w_2\right)\right|$ implies that 
\[ \left|\phi\left(\frac{\lambda_{j_1}+\cdots+\lambda_{j_m}}m\right)\right|< \left|\phi\left(\lambda_{j_1}\right)\right|^{1/m}\cdots  \left|\phi\left(\lambda_{j_m}\right)\right|^{1/m}.\]
Moreover, 
\[ |c_\bj(N)| \times \left|\phi\left(\frac{\lambda_{j_1}+\cdots+\lambda_{j_m}}m\right)\right|^N \leq |b_\bj|^{1/m}\left|
\frac{\phi\left(\frac{\lambda_{j_1}+\cdots+\lambda_{j_m}}m\right)}{\phi\left(\lambda_{j_1}\right)^{1/m}\cdots \phi\left(\lambda_{j_m}\right)^{1/m}}\right|^N.
\]
Since the left hand side of this inequality goes to zero, we get that $T^N v_2$ tends to zero.
We conclude the proof by observing that 
$$T^N v_3=\sum_{j=1}^q b_j E\left(\lambda_j \right).$$
 \end{proof}
 
For the applications, we emphasize two corollaries of Theorem \ref{thm:philarge}.

\begin{corollary}\label{cor:smalleigengeneralbis}
Let $X$ be an F-algebra and let $T\in\mathcal L(X)$. Assume that there exist a function $E:\CC\to X$ and
an entire function $\phi:\CC\to \CC$ satisfying the following assumptions:
\begin{enumerate}
\item for all $\lambda\in\CC$,\ $TE(\lambda)=\phi(\lambda)E(\lambda)$;
\item for all $\lambda,\mu\in\CC$, $E(\lambda)E(\mu)=E(\lambda+\mu)$;
\item for all $\Lambda\subset\CC$ with an accumulation point, $\textrm{span}\left(E(\lambda);\ \lambda\in\Lambda\right)$ is dense in $X$;
\item $\phi$ is not a multiple of an exponential function;
\item for all $\rho\in (0,1)$, there exists $w_0\in\CC$ with $\left|\phi\left(w_0\right)\right|>1$ and, for all $r\in (0,\rho]$, $\left|\phi\left(rw_0\right)\right|<1$.
\end{enumerate}
Then $T$ supports a hypercyclic algebra.
\end{corollary}
\begin{proof}
We show that Assumption (5) of Theorem \ref{thm:philarge} is satisfied. Let $m\geq 1$ and $\veps\in(0,1/m)$. 
Set $\rho=\frac{m-1}m+m\veps$. We get the existence of $w_0$. We set $b=w_0/m$ and $a=\veps w_0/m$. Then, for any $d\leq m-1$ and any $n\leq m$, 
$$|db+(n-d)a|\leq \left(\frac{m-1}m+\veps\right)|w_0|\leq \rho |w_0|$$
showing that $|\phi(db+(n-d)a)|<1$.
\end{proof}

\begin{corollary}\label{cor:smalleigengeneral}
Let $X$ be an F-algebra and let $T\in\mathcal L(X)$. Assume that there exist a function $E:\CC\to X$ and
an entire function $\phi:\CC\to \CC$ satisfying the following assumptions:
\begin{enumerate}
\item for all $\lambda\in\CC$,\ $TE(\lambda)=\phi(\lambda)E(\lambda)$;
\item for all $\lambda,\mu\in\CC$, $E(\lambda)E(\mu)=E(\lambda+\mu)$;
\item for all $\Lambda\subset\CC$ with an accumulation point, $\textrm{span}\left(E(\lambda);\ \lambda\in\Lambda\right)$ is dense in $X$;
\item $\phi$ is not a multiple of an exponential function;
\item $|\phi(0)|<1$.
\end{enumerate}
Then $T$ supports a hypercyclic algebra.
\end{corollary}

\begin{proof}
We prove that Assumption (5) of Corollary  \ref{cor:smalleigengeneralbis} is satisfied. 
Denote by $M(r)=\sup\{|\phi(z)|;\ |z|=r\}$ which is a continuous and increasing function of $r$ satisfying $M(0)<1$.
Let $r_0>0$ with $M(r_0)=1$ and let $r_1>r_0$ with $\rho r_1<r_0$. Any $w_0\in\CC$ such that $|w_0|=r_1$ and 
$|\phi(w_0)|=M(r_1)>1$ does the job.
\end{proof}

\begin{remark}
 An operator on a Banach space cannot satisfy the assumptions of Theorem \ref{thm:philarge}. Indeed, they imply that its spectrum is unbounded.
 We shall see later how it remains possible to get a hypercyclic algebra in a Banach algebra context.
\end{remark}

\subsection{Applications to convolution operators}

We now show how to deduce the first half of Theorem \ref{thm:mainconvolution} from Corollary \ref{cor:smalleigengeneral}. Thus, let $\phi$ be an entire function of exponential type which is not a multiple of an exponential function. Then we let $X=H(\CC)$ and $T=\phi(D)$.
The map $E$ is defined by $E(\lambda)(z)=\exp(\lambda z)$; it satisfies (1), (2) and (3) of Corollary \ref{cor:smalleigengeneral}.

Corollary \ref{cor:smalleigengeneralbis} may also be applied to functions satisfying $|\phi(0)|=1$. We give here two examples
which are not covered by Theorem \ref{thm:mainconvolution}.

\begin{example}(\cite{BCP18})
Let $\phi(z)=\cos(z)$. Then $\phi(D)$ supports a hypercyclic algebra.
\end{example}
\begin{proof}
 Recall that if $z=x+iy$, then $|\cos(z)|^2=\cos^2 x+\sinh^2 y$. Let us set $\psi(t)=\cos^2(2t)+\sinh^2(t)$. Using 
 standard calculus one may prove that there exists $t_0>0$ such that $\psi$ is decreasing on $(0,t_0)$ and increasing on $(t_0,+\infty)$.
 Let $t_1>0$ be such that $\psi(t_1)=1$. It then suffices to consider $w_0=(1+\eta)(2+i)t_1$ for some sufficiently small $\eta>0$. 
\end{proof}
 
\begin{example}
Let $\phi(z)=e^z-2$. Then $\phi(D)$ supports a hypercyclic algebra.
\end{example}
\begin{proof}
Let $t_1>0$ be  such that $\phi(t_1)=1$. It suffices to consider $w_0=(1+\eta)t_1$ for some sufficiently small $\eta>0$.
\end{proof}

We now give a surprizing example of an entire function $\phi$ with $|\phi(0)|>1$ and $\phi(D)$ supports a hypercyclic algebra.

\begin{example}\label{ex:philarge}
Let $\phi(z)=2e^{-z}+\sin z$. Then $\phi(D)$ supports a hypercyclic algebra. 
\end{example}
\begin{proof}
We just need to show that $\phi$ satisfies Assumption (5) of Theorem \ref{thm:philarge}. We let $a=k\pi$ for some sufficiently large $k$ and $b=k\pi+\frac{\pi}{2m}$. Then 
$$\left|\phi\left(db+(n-d)a\right)\right|=\left|\sin\left(nk\pi+\frac{d\pi}{2m}\right)+2\exp\left(-nk\pi-\frac{d\pi}{2m}\right)\right|.$$
Provided $k$ is large enough, this is less than $1$ as soon as $d<m$, whereas $|\phi(mb)|>1$.
\end{proof}

\subsection{Applications to composition operators}
In this subsection, we prove Theorem \ref{thm:maincomposition}. Recall that given a simply connected domain 
$\Omega\subset\CC$ and $\varphi$ a holomorphic self-map of $\Omega$, $C_\varphi$ is hypercyclic if and only if $\varphi$ is univalent and without fixed points. We first prove Theorem \ref{thm:maincomposition} when $\Omega=\CC$.
In that case $\varphi$ is also entire hence $\varphi$ is a translation $\varphi(z)=z+a$, $a\neq 0$. Thus $P(C_\varphi)=\phi(D)$, where $\phi(z)=P\circ \exp(az)$ and the result is a particular case of Theorem \ref{thm:mainconvolution}.

Otherwise, by the Riemann mapping theorem, we may assume that $\Omega=\DD$. We simplify the proof by using the linear fractional model (see for instance \cite{BoSh97}): since $\varphi$ has no fixed points in $\DD$, there exists a univalent map $\sigma:\DD\to\CC$ and a linear fractional map $\psi$ such that $\sigma\circ\varphi=\psi\circ\sigma$. Moreover
\begin{itemize}
\item either $\psi$ can be taken to be a dilation $\psi(z)=rz$ for some $0<r<1$;
\item or $\psi$ can be taken to be a translation $\psi(z)=z+a$ for some $a\in\CC\backslash\{0\}$ and $\sigma(\DD)\subset \{z;\ \Re e(z)>0\}$. 
\end{itemize}
The functional equation guarantees that $\sigma(\DD)=:\mathcal U$ is preserved by $\psi$ and that $P(C_\varphi)$ acting on $H(\DD)$ and $P(C_\psi)$ acting on $H(\mathcal U)$ are quasi-conjugate by $C_\sigma$.
Since $C_\sigma$ is a multiplicative map, it is sufficient to prove that $P(C_\psi)$ admits a hypercyclic algebra 
(see \cite[Remark 6]{BCP18}).

\noindent $\blacktriangleright$ The translation case. We denote $T=P(C_\psi)$ and $E(\lambda)(z)=\exp(\lambda z)$ so that $TE(\lambda)=\phi(\lambda)E(\lambda)$ with
$\phi(\lambda)=P(\exp(a\lambda))$. Then the assumptions of Corollary \ref{cor:smalleigengeneral} are satisfied provided we are able to prove that, for any $\Lambda\subset\CC$ with an accumulation point, 
$\textrm{span}\left(E(\lambda);\ \lambda\in\Lambda\right)$ is dense in $H(\mathcal U)$. The proof is exactly similar to that for $H(\CC)$, since the polynomials are dense in $H(\mathcal U)$ (recall that $\mathcal U$ is simply connected) - see for instance \cite{GoSh91}.

\noindent $\blacktriangleright$ The dilation case.  We still denote $T=P(C_\psi)$ but now we set $E(\lambda)=z^\lambda$. This defines a holomorphic function on $\mathcal U$ since $\mathcal U\subset\{z;\ \Re e(z)>0\}$. Moreover, $TE(\lambda)=\phi(\lambda)E(\lambda)$ with $\phi(\lambda)=P\left(\exp(\lambda\log r)\right)$. Again the assumptions of Corollary \ref{cor:smalleigengeneral}  are satisfied provided that, for all $\Lambda\subset\CC$ with an accumulation point, $\textrm{span}\left(E(\lambda);\ \lambda\in\Lambda\right)$ is dense in $H(\mathcal U)$. Let $L$ be a linear form on $H(\mathcal U)$ which vanishes on $\textrm{span}\left(E(\lambda);\ \lambda\in\Lambda\right)$. By the Riesz representation theorem, there exists $K$ a compact subset of $\mathcal U$ and $\mu$ a complex measure supported in $K$ such that, for all $f\in H(\mathcal U)$, 
$L(f)=\int_K fd\mu$. The map $\lambda\mapsto L(z^\lambda)$ is holomorphic and has an accumulation point
of zeros. Hence it is identically zero. Therefore, $L$ vanishes on each monomial $z^n$, hence on $H(\mathcal U)$ since $\mathcal U$ is simply connected. 
This shows that $\textrm{span}\left(E(\lambda);\ \lambda\in\Lambda\right)$ is dense in $H(\mathcal U)$.


\section{Large eigenvalues}\label{sec:largeeigenvalues}

As in Section \ref{sec:smalleigenvalues}, we shall deduce Part (2) of Theorem \ref{thm:mainconvolution} from a more general statement.

\begin{theorem}\label{thm:largeeigengeneral}
Let $X$ be an F-algebra and let $T\in\mathcal L(X)$. Assume that there exist a function $E:\CC\to X$ and
a nonconstant entire function $\phi:\CC\to \CC$ satisfying the following assumptions:
\begin{enumerate}
\item for all $\lambda\in\CC$,\ $TE(\lambda)=\phi(\lambda)E(\lambda)$;
\item for all $\lambda,\mu\in\CC$, $E(\lambda)E(\mu)=E(\lambda+\mu)$;
\item for all $\Lambda\subset\CC$ with an accumulation point, $\textrm{span}\left(E(\lambda);\ \lambda\in\Lambda\right)$ is dense in $X$;
\item $|\phi(0)|=1$, $\phi$ has subexponential growth and either $\phi'(0)\neq 0$ or $\phi$ has order less than $1/2$. 
\end{enumerate}
Then $T$ supports a hypercyclic algebra.
\end{theorem}

The proof of Theorem \ref{thm:largeeigengeneral} shares many similarities with that of Theorem \ref{thm:philarge}.
Nevertheless, we will now choose the complex numbers $\lambda_j$ with $\left|\phi\left(\lambda_j\right)\right|$ very large
(instead of being slightly bigger than 1). In this way, because $\phi$ has subexponential growth, we will ensure that $|\phi(\lambda_j)|$ is bigger than $|\phi(2\lambda_j)|^{1/2}$.
Thus, when we will take the powers of $u$ and apply $T^N$, the main term will change. We will also need a more careful
interaction between the $\lambda_j$' and the $\gamma_k$'. This is the content of the following key lemma, which uses the fact that we control the growth of $\phi$.

\begin{lemma}\label{lem:order}
 Let $\phi$ be a nonconstant entire function with subexponential growth and $|\phi(0)|=1$. Assume that either $\phi'(0)\neq 0$ or $\phi$ has order less than $1/2$. 
 Then for all $m\geq 2$, there exist $z_0\in\CC\backslash\{0\}$ and $w_0= \rho z_0$ for some $\rho>0$ such that  
 \begin{itemize}
 \item $|\phi|<1$ on $(0,z_0]$;
  \item $|\phi(w_0)|>1$;
  \item $|\phi(w_0)|>|\phi(dw_0)|^{1/d}$ for all $d=2,\dots,m$.
  \item $t\mapsto |\phi(w_0+tz_0)|$ is increasing on some interval $[0,\eta)$, $\eta>0$. 
 \end{itemize}
\end{lemma}
\begin{proof}
We first show the existence of $z_0,z_1\in\CC$ with $z_0\in (0,z_1)$, $|\phi|<1$ on $(0,z_0]$ and $|\phi(z_1)|>1$.
The proof differs here following the assumptions made on $\phi$. Assume first that $\phi$ has order less than $1/2$. 
Write $\phi(z)=e^{i\theta_0}+\rho_p e^{i\theta_p}z^p+o(z^p)$ with $\rho_p>0$. Then, $\phi\left(te^{-i(\theta_p-\theta_0+\pi)/p}\right)=e^{i\theta_0}-\rho_p t^pe^{i\theta_0}+o(t^p)$
 has modulus less than $1$ provided $t$ is small enough. We then set $z_0=te^{-i(\theta_p-\theta_0+\pi)/p}$ for some small $t$. We then find $z_1$ since
 any nonconstant entire function of order less than 1/2 cannot be bounded
on a half-line (see \cite[Theorem 3.1.5]{Boas54}).

On the other hand, suppose now that $\phi'(0)\neq 0$. Without loss of generality we may assume that $\phi(z)=1-az+o(z)$ for some $a>0$. 
Now, a nonconstant entire function with subexponential growth cannot be bounded on a half-plane (see \cite[Theorem 1.4.3]{Boas54}).
Thus, there exists $z_1=r_1e^{i\alpha_1}$ with $r_1>0$ and $\alpha_1\in (-\pi/2,\pi/2)$ such that $|\phi(z_1)|>1$. It is easy to check that, for $t>0$ small enough,
$|\phi(tz_1)|<1$ and we set $z_0=tz_1$ for such a small $t>0$. 

We now proceed with the construction of $w_0$. 
 Without loss of generality we may assume $z_0=1$. We set $\psi(t)=|\phi(t)|^2$. We proceed by contradiction and we assume
 that there does not exist $w_0>0$ such that the last three points of the lemma are satisfied. 
 We fix a sequence $(\veps_n)$ in $(0,1)$ such that $\kappa:=\prod_{n=1}^{+\infty}(1-\veps_n)>0$. 
We shall construct two sequences $(t_n)_{n\geq 0}$ and $(r_n)_{n\geq 1}$ 
 of positive real numbers such that, for all $n\geq 0$, 
 $$\left\{ 
 \begin{array}{l}
 \psi(t_n)>1,\ \psi'(t_n)>0\\
 t_n=r_{n}t_{n-1}\\
 r_n\in (1,m]\\
 \psi(t_n)\geq \max\left(\psi(t_{n-1})^{(1-\veps_n)r_n},\psi(t_{n-1})^{3/2}\right).
 \end{array}
 \right.$$
First, the existence of $z_1$ leads to some positive real number $t_0>1$ such that 
$\psi(t_0)>1$ and $\psi'(t_0)>0$. Next, assume that the construction has been done until step $n$ and let us proceed with step $n+1$.
Since we assumed that the conclusion of Lemma \ref{lem:order} is false, there exists $k_{n+1}\in \{2,\dots,m\}$ such that 
$$\psi(k_{n+1}t_n)\geq \psi(t_n)^{k_{n+1}}>\max\left(\psi(t_n)^{(1-\veps_{n+1})k_{n+1}},\psi(t_n)^{3/2}\right).$$
If $\psi'(k_{n+1}t_n)>0$, then we are done by choosing $t_{n+1}=k_{n+1}t_n$. Otherwise, let
$$\tau:=\sup\{t\in [t_n,k_{n+1}t_n];\ \psi'(t)>0\}.$$
Then
\[\psi(\tau)\geq \psi(k_{n+1}t_n).\]
Thus, there exists $t_{n+1}\in[t_n,\tau]$ such that $\psi'(t_{n+1})>0$ and
\[\psi(t_{n+1})>\max\left(\psi(t_n)^{(1-\veps_{n+1})k_{n+1}},\psi(t_n)^{3/2}\right).\]
Now, $t_{n+1}=r_{n+1}t_n$ for some $r_{n+1}\in (1,k_{n+1})$ so that
\[\psi(t_{n+1})>\max\left(\psi(t_n)^{(1-\veps_{n+1})r_{n+1}},\psi(t_n)^{3/2}\right)\]
as required to prove step $n+1$. The sequence $(t_n)$ we have just built satisfies, for all $n\leq N$,
\[\psi(t_n)\geq \left(\psi(t_0)\right)^{(3/2)^n}.\]
In particular, $\left(\psi\left(t_n\right)\right)$, hence $\left(t_n\right)$, go to infinity. Now
$$\psi(t_n)\geq \psi(t_0)^{\prod_{k=1}^n (1-\veps_k)r_k}\geq \psi(t_0)^{\kappa t_n/t_0}.$$

This is a contradiction since $\psi$ has subexponential growth.
\end{proof}

\begin{proof}[Proof of Theorem \ref{thm:largeeigengeneral}]
Let $w_0,z_0\neq 0$ be given by Lemma \ref{lem:order}.
 Then we may find $\gamma_1\in (0,z_0/m)$ which is sufficiently close to 0 so that
 \begin{itemize}
  \item $\left|\phi\left(w_0+(m-1)\gamma_1\right)\right|>1$;
  \item $\left|\phi\left(w_0+(m-1)\gamma_1\right)\right|>\left|\phi\left(dw_0+s\gamma_1\right)\right|^{1/d}$ for all $d\in\{2,\dots,m\}$ and for all $s\in \{0,\dots,m-d\}$;
  \item $\left|\phi\left(w_0+(m-1)\gamma_1\right)\right|>\left|\phi\left(w_0+s\gamma_1\right)\right|$ for all $s\in\{0,\dots,m-2\}$.
 \end{itemize}
 We then fix $\delta>0$ sufficiently small so that
\begin{itemize}
 \item for all $\lambda\in B(w_0,\delta)$, $\left|\phi\left(\lambda+(m-1)\gamma_1\right)\right|>1$;
 \item for all $d\in\{1,\dots,m\}$, for all $s\in\{0,\dots,m-d\}$ with $(d,s)\neq (1,m-1)$, for all $\lambda,\lambda_1,\dots\lambda_d\in B(w_0,\delta)$,
 for all $z\in B(0,\delta)$, 
 $$\left|\phi\left(\lambda+(m-1)\gamma_1\right)\right|>\left|\phi\left(\lambda_1+\cdots+\lambda_d+s\gamma_1+z\right)\right|^{1/d}.$$
\end{itemize}
Let $p,q$ be integers, let $a_1,\dots,a_p,b_1,\dots,b_q$ be complex numbers, let $\gamma_2,\dots,\gamma_p\in (0,z_0/m)\cap B(0,\delta/m)$ and let
$\lambda_1,\dots,\lambda_q\in B(w_0,\delta)$ be such that
$$\sum_{l=1}^p a_l E\left(\gamma_l \right)\in U$$
$$\sum_{j=1}^q b_j E\left(\lambda_j+(m-1)\gamma_1\right) \in V.$$
Without loss of generality we may assume $a_1\neq 0$. For $N\geq 1$ and $j\in\{1,\dots,q\}$, let $c_j:=c_j(N)$ be defined by
$$c_j=\frac{b_j}{ma_1^{m-1}\big(\phi\left(\lambda_j+(m-1)\gamma_1\right)\big)^N}$$
and let us set
$$u:=u_N=\sum_{l=1}^p a_l E\left(\gamma_l \right)+\sum_{j=1}^q c_j E\left(\lambda_j \right)$$
(observe that now we do not divide $\lambda_j$ by $m$).
As before, for $N$ large enough, $u$ belongs to $U$. Moreover, the formula for $u^n$ is similar:
$$u^n=\sum_{d=0}^n \sum_{\substack{\bl\in I_p^{n-d} \\\bj\in I_q^d}}\alpha(\bl,\bj,d,n)a_\bl c_\bj E\left(\gamma_{l_1}+\cdots+\gamma_{l_{n-d}}+\lambda_{j_1}+\cdots+\lambda_{j_d}\right).$$
Assume first that $n<m$ and let us show that, for all $d\in\{0,\dots,n\}$, for all $\bl\in I_p^{n-d}$ and all $\bj\in I_q^d$,
\begin{equation}\label{eq:main2} 
\left|c_\bj(N)\right|\times\left|\phi\left(\gamma_{l_1}+\cdots+\gamma_{l_{n-d}}+\lambda_{j_1}+\cdots+\lambda_{j_d}\right)\right|^N\longrightarrow_{N\to+\infty}0.
\end{equation}

Assume first $d\geq 1$ and let $s=\card\left\{i;\ l_i=1\right\}$. Then $\gamma_{l_1}+\cdots+\gamma_{l_{n-d}}+\lambda_{j_1}+\cdots+\lambda_{j_d}=\lambda_{j_1}+\cdots+\lambda_{j_d}+s\gamma_1+z$
with $|z|<\delta$ and $s\leq m-2$. So, writing
\begin{align*}
 &\left|c_\bj(N)\right|\times\left|\phi\left(\gamma_{l_1}+\cdots+\gamma_{l_{n-d}}+\lambda_{j_1}+\cdots+\lambda_{j_d}\right)\right|^N\\
&\quad =\frac{|b_\bj|}{m^d|a_1|^{(m-1)d}}\left(\prod_{i=1}^d \frac{\left|\phi\left(\lambda_{j_1}+\cdots+\lambda_{j_d}+s\gamma_1+z\right)\right|^{1/d}}{\left|\phi\left(\lambda_{j_i}+(m-1)\gamma_1\right)\right|}\right)^N
\end{align*}
we observe that \eqref{eq:main2} is true. If $d=0$ (in that case, $c_\bj(N)=1$), \eqref{eq:main2} remains also true since $\gamma_{l_1}+\cdots+\gamma_{l_n}\in (0,z_0]$ so that 
$\left|\phi\left(\gamma_{l_1}+\cdots+\gamma_{l_n}\right)\right|<1$. This yields that $T^N (u^n)$ tends to zero. The case $n=m$ requires small modifications. 
We now decompose $u^m$ into 
\begin{align*}
 u^m&=ma_1^{m-1}\sum_{j=1}^q c_j E\left(\lambda_j+(m-1)\gamma_1\right)\\
 &\quad+\sum_{\substack{d=0\\d\neq 1}}^m \sum_{\substack{\bl\in I_p^{m-d} \\\bj\in I_q^d}}\alpha(\bl,\bj,d,m)a_\bl c_\bj E\left(\gamma_{l_1}+\cdots+\gamma_{l_{n-d}}+\lambda_{j_1}+\cdots+\lambda_{j_d}\right)\\
 &\quad+\sum_{\substack{\bl\in I_p^{m-1} \\\bj\in I_q^1\\ \bl\neq (1,\dots,1)}}\alpha(\bl,\bj,d,m)a_\bl c_\bj E\left(\gamma_{l_1}+\cdots+\gamma_{l_{n-d}}+\lambda_{j_1}+\cdots+\lambda_{j_d}\right)\\
 &=:v_1+v_2+v_3.
\end{align*}
With exactly the same argument as above, one shows that $T^N(v_2+v_3)$ tends to zero. Furthermore,
$$T^N(v_1)=\sum_{j=1}^m b_j E\left(\lambda_j+(m-1)\gamma_1\right)\in V,$$
which closes the argument.
\end{proof}

\section{Unimodular eigenvalues}\label{sec:unimodular}
\subsection{Proof of Theorem \ref{thm:mainbackward}}
In this section, we provide a proof for Theorem \ref{thm:mainbackward}. As before, $P(B)$ admits a natural family of eigenvectors: 
for any $\lambda\in\DD$, $(\lambda^k)$ is an eigenvector of $P(B)$ associated to $P(\lambda)$. Nevertheless, we do not have a so simple formula for the product of two eigenvectors. 

\begin{lemma}\label{lem:producteigenvectors}
Let $\Theta=(\theta_1,\dots,\theta_n)\in\CC^n$ and let $\mu_1,\dots,\mu_r$ be pairwise distinct complex numbers such that $\{\theta_1,\dots,\theta_n\}=\{\mu_1,\dots,\mu_r\}$. 
For $j=1,\dots,r$, let $\kappa_j=\card\{k;\ \theta_k=\mu_j\}$. Then there exist polynomials $P_{\Theta,j}$, $1\leq j\leq r$, with degree less than or equal to $\kappa_j-1$ such  that
$$\left(\theta_1^k\right)\star\cdots\star\left(\theta_n^k\right)=\sum_{j=1}^r \left(P_{\Theta,j}(k)\mu_j^k\right).$$
When all the $\theta_i$ are equal to the same $\lambda$, then 
$$\left(\lambda^k\right)\star\cdots\star\left(\lambda^k\right)=\left(P_n(k)\lambda^k\right)$$
where $\deg(P_n)=n-1$.
\end{lemma}

The statement of this lemma motivates the study of the effect of $P(B)^N$ on the vectors $(k^d\lambda^k)$.
\begin{lemma}\label{lem:effect}
Let $P\in\CC[X]$ and let $d\geq 0$. Let also $\lambda\in\DD$ with $\lambda P(\lambda)P'(\lambda)\neq 0$. There exist complex numbers $(A_{d,N,s})_{N\geq 0,\ 0\leq s\leq d}$, such that, for all $N\geq 0$, 
\begin{equation}
\big(P(B)\big)^N\left(k^d\lambda^k\right)=\sum_{s=0}^d P(\lambda)^{N+s-d}A_{d,N,s}\left(k^s\lambda^k\right)\label{eq:effect}
\end{equation}
with $A_{d,N,s}\sim_{N\to+\infty}\omega_{d,s}N^{d-s}$ for some nonzero $\omega_{d,s}$.
\end{lemma}
We point out that in the statement of the previous lemma, the complex number $\omega_{d,s}$ may depend on $\lambda$; later we will sometimes denote them $\omega_{d,s}(\lambda)$.

We will also need a density lemma.

\begin{lemma}\label{lem:density}
Let $\Lambda\subset\DD$ with an accumulation point inside $\DD$. Then $\left\{\left(\lambda^k\right);\ \lambda\in\Lambda\right\}$ spans a dense subspace of  $\ell^1(\NN)$. 
\end{lemma}

We postpone the proof of these lemmas to give that of Theorem \ref{thm:mainbackward}.

\begin{proof}[Proof of Theorem \ref{thm:mainbackward}]
Since the spectrum of a hypercyclic operator has to intersect the unit circle, if $P(D)$ is hypercyclic, then $P(\DD)\cap\TT\neq\varnothing$. Hence the only difficult implication is $(iii)\implies (ii)$. Thus we start with a nonconstant polynomial satisfying $P(\DD)\cap\TT\neq\varnothing$. Let $\Lambda_1\subset\DD$ with an accumulation point in $\DD$ such that $|P(\lambda)|=1$ and $\lambda P'(\lambda)\neq 0$ for all $\lambda\in\Lambda_1$.
Let also $\Lambda_2\subset\DD$ with an accumulation point in $\DD$ such that $|P(\lambda)|<1$ and $\lambda P'(\lambda)\neq 0$ for all $\lambda\in\Lambda_2$.

Let $U,V,W$ be nonempty open subsets of $\ell^1$ with $0\in W$. Let $m\geq 1$. We may find $p,q\in\NN$, complex numbers $\gamma_1,\dots,\gamma_p\in\Lambda_2$, $\lambda_1,\dots,\lambda_q\in\Lambda_1$, $a_1,\dots,a_p$, $b_1,\dots,b_q$ such that 
$$\sum_{l=1}^p a_l \left(\gamma_l^k\right)\in U\textrm{ and }\sum_{j=1}^q b_j\left(\lambda_j^k\right)\in V.$$
We then set, for $j=1,\dots,q$, $N\geq 0$, 
$$c_j:=c_j(N)=\left(\frac{b_j}{\omega_{m-1,0}(\lambda_j)N^{m-1}P(\lambda_j)^{N-m+1}}\right)^{1/m}$$
(we take any $m$-th root) and 
$$u:=u_N=\sum_{l=1}^p  a_l \left(\gamma_l^k\right)+\sum_{j=1}^q c_j\left(\lambda_j^k\right)$$
so that, if $N$ is large enough, $u$ belongs to $U$. As usual, for $n\in\NN$,
$$u^n=\sum_{d=0}^n \sum_{\substack{\bl\in I_p^{n-d} \\\bj\in I_q^d}}\alpha(\bl,\bj,d,n)a_\bl c_\bj 
\left(\gamma_{l_1}^k\right)\star\cdots\star\left(\gamma_{l_{n-d}}^k\right)\star\left(\lambda_{j_1}^k\right)\star\cdots\star\left(\lambda_{j_d}^k\right).$$
Let us fix for a while $n\leq m$, $d\in\{0,\dots,n\}$, $\bl\in I_p^{n-d}$ and $\bj\in I_q^d$. Applying Lemma \ref{lem:producteigenvectors}, we observe that $\left(\gamma_{l_1}^k\right)\star\cdots\star\left(\lambda_{j_d}^k\right)$ writes as a linear combination of $\left(k^s\mu^k\right)$ for some $\mu\in\left\{\gamma_{l_1},\dots,\lambda_{j_d}\right\}$ and 
$s\leq \card\left\{l;\ \gamma_l=\mu\right\}+\card\left\{j;\ \lambda_j=\mu\right\}-1.$ Moreover, by Lemma \ref{lem:effect}, 
\[ \left\|P(B)^N \left(k^s\mu^k\right)\right\|\leq C N^s |P(\mu)|^N.\]
In particular, if $\mu\in\left\{\gamma_{l_1},\dots,\gamma_{l_{n-d}}\right\}$, so that $|P(\mu)|<1$, then $|c_\bj(N)|\times \left\|P(B)^N \left(k^s\mu^k\right)\right\|$ tends to zero. If $\mu\in\left\{\lambda_{j_1},\dots,\lambda_{j_d}\right\}$, then 
\begin{equation}
\label{eq:backward1}
|c_\bj(N)|\times \left\|P(B)^N \left(k^s\mu^k\right)\right\|\leq C\frac{N^s}{N^{d\times\frac {m-1}m}}.
\end{equation}
Assume first that $n<m$. Then
\[\frac sd\leq\frac{d-1}d\leq\frac{n-1}n<\frac{m-1}m.\]
Hence, the right hand side of \eqref{eq:backward1} goes to zero as $N$ tends to $+\infty$. This implies in particular that $P(B)^N (u^n)$ goes to zero as $N$ tends to $+\infty$. 
Assume now that $n=m$. The same argument shows that $|c_\bj(N)|\times \left\|P(B)^N \left(k^s\mu^k\right)\right\|$ tends to zero, except if $s=d-1=m-1$. Namely, $c_\bj(N)P(B)^N\left(\left(\gamma_{l_1}^k\right)\star\cdots\star\left(\gamma_{l_{n-d}}^k\right)\star\left(\lambda_{j_1}^k\right)\star\cdots\star\left(\lambda_{j_d}^k\right)\right)$ tends to zero except for the terms $c_\bj(N)P(B)^N\left(\left(\lambda_j^k\right)\star\cdots\star\left(\lambda_j^k\right)\right)$. Applying again Lemma \ref{lem:producteigenvectors} and Lemma \ref{lem:effect}, we find that
$$P(B)^N\left(\left(\lambda_j^k\right)\star\cdots\star\left(\lambda_j^k\right)\right)=P(\lambda)^{N-m+1} \omega_{m-1,0}(\lambda_j)N^{m-1}(\lambda_j^k)+o(N^{m-1}).$$
Hence, by the definition of $c_j$, 
$$c_\bj P(B)^N\left(\left(\lambda_j^k\right)\star\cdots\star\left(\lambda_j^k\right)\right)=b_j(\lambda_j^k)+o(1).$$
This achieves the proof that $P(B)^N(u^m)$ belongs to $V$ provided $N$ is large enough.
\end{proof}

\subsection{Proof of Lemma \ref{lem:producteigenvectors}}

The proof of Lemma \ref{lem:producteigenvectors} relies on the following facts and an easy induction.

\noindent{\textbf{Fact 1.}} For any $d\geq 0$, for any $\lambda\in \DD$, there exists a polynomial $Q_d$ with $\deg(Q_d)=d+1$ such that 
\[ \left(k^d\lambda^k\right)\star \left(\lambda^k\right)=\left(Q_d(k)\lambda^k\right). \]

\noindent{\textbf{Fact 2.}} For any $d\geq 0$, for any $\lambda,\mu\in\DD$ with $\lambda\neq\mu$, there exist a polynomial $Q_{d,\lambda,\mu}$ with $\deg(Q_{d,\lambda,\mu})\leq d$
and a complex number $B_{d,\lambda,\mu}$ such that 
 \[ \left(k^d \lambda^k\right)\star \left(\mu^k\right)=\left(Q_{d,\lambda,\mu}(k)\lambda^k\right)+B_{d,\lambda,\mu}\left(\mu^k\right). \]

The proof of Fact 1 is easy. Denoting $\left(k^d\lambda^k\right)\star\left(\lambda^k\right)$ by $(u_k)$, we have
$$u_k=\sum_{j=0}^k j^d\lambda^k$$
so that the result is proved with $Q_d$ the polynomial of degree $d+1$ such that $Q_d(k)=\sum_{j=0}^k j^d$ for all $k\in\NN$.

The proof of Fact 2 can be done by induction on $d$. For $d=0$, we simply write $\left(\lambda^k\right)\star\left(\mu^k\right)$ as $\left(u_k\right)$ with 
\[ u_k=\sum_{j=0}^k \lambda^j\mu^{k-j}=\frac{\lambda}{\lambda-\mu}\left(\lambda^k\right)+\frac{\mu}{\mu-\lambda}\left(\mu^k\right).\]
For the induction step, we write using Fact 1
\[ \left(k^{d+1}\lambda^k\right)=a\left(k^d\lambda^k\right)\star\left(\lambda^k\right)+\left(P(k)\lambda^k\right) \]
with $\deg(P)\leq d$ (to simplify the notations, we do not write the subscripts on the complex numbers 
and on the polynomials involved, but they clearly depend on $d$, $\lambda$, and later on $\mu$). Thus,
\[ \left(k^{d+1}\lambda^k\right)\star\left(\mu^k\right)=a\left(\lambda^k\right)\star\left(\left(k^d \lambda^k\right)\star \left(\mu^k\right)\right)+\left(P(k)\lambda^k\right)\star\left(\mu^k\right). \]
We then apply the induction hypothesis to both $\left(k^d\lambda^k\right)\star \left(\mu^k\right)$ and 
$\left(P(k)\lambda^k\right)\star\left(\mu^k\right)$ to get
\[ \left(k^{d+1}\lambda^k\right)\star\left(\mu^k\right)=a\left(\lambda^k\right)\star \left(\left(Q(k)\lambda^k\right)+b\left(\mu^k\right)\right)+\left(R(k)\lambda^k\right)+c\left(\mu^k\right)\]
with $\deg(Q),\deg(R)\leq d$ and $b,c\in\CC$. We conclude by using again either Fact 1 or the case $d=0$.

\subsection{Proof of Lemma \ref{lem:effect}}

We first isolate the case $N=1$. 
\begin{lemma}\label{lem:effectauxiliary}
Let $P\in\CC[X]$, $d\geq 0$. There exist polynomials $(Q_{d,s})$, $0\leq s\leq d$, such that, for all $\lambda\in\DD$, 
\begin{equation*}
P(B)\left(k^d\lambda^k\right)=\sum_{s=0}^d Q_{d,s}(\lambda)\left( k^s\lambda^k\right)
\end{equation*}
where $Q_{d,d}(\lambda)=P(\lambda)$ and $Q_{d,d-1}(\lambda)=d \lambda P'(\lambda)$.
\end{lemma}
\begin{proof}
We start from 
\[B^n\left(k^d\lambda^k\right)=\left((k+n)^d\lambda^{k+n}\right)=\lambda^n \left(P_{n,d}(k)\lambda^k\right)\]
where $P_{n,d}$ is a monic polynomial of degree $d$. More precisely, $P_{n,d}(k)=k^d+dnk ^{d-1}+\cdots$. The result follows now from a linear combination of these equalities. Observe that if $P(X)=\sum_{j=0}^n\alpha_j X^j$, then $Q_{d,d-1}(\lambda)=\sum_{j=1}^n \alpha_j dj\lambda^j=d\lambda P'(\lambda)$.
\end{proof}
\begin{proof}[Proof of Lemma \ref{lem:effect}]
We first prove by induction on $N$ that the relation \eqref{eq:effect}, which is clearly true for $N=0$, holds for all $N$ and we get an induction formula for the complex numbers $A_{d,N,s}$. Indeed, assuming \eqref{eq:effect} is true for $N$ and using Lemma \ref{lem:effectauxiliary}, we have
\begin{align*}
\left(P(B)\right)^{N+1}(k^d\lambda^k)&=\sum_{r=0}^d P(\lambda)^{N+r-d}A_{d,N,r}\sum_{s=0}^r Q_{r,s}(\lambda) \left(k^s\lambda^k\right)\\
&=\sum_{s=0}^d \sum_{r=s}^d P(\lambda)^{N+r-d}A_{d,N,r}Q_{r,s}(\lambda)\left(k^s\lambda^k\right).
\end{align*}
Thus, \eqref{eq:effect} is true for $N+1$ with the induction formula
\begin{equation}\label{eq:inductionformula}
A_{d,N+1,s}=\sum_{r=s}^d P(\lambda)^{r-s-1}A_{d,N,r}Q_{r,s}(\lambda).
\end{equation}
When $s=d$, using $Q_{d,d}(\lambda)=P(\lambda)$, this formula simply writes $A_{d,N+1,d}=A_{d,N,d}$ so that $A_{d,N,d}=1$ for all $N$. When $s=d-1$, using $Q_{d-1,d-1}(\lambda)=P(\lambda)$ and $Q_{d,d-1}(\lambda)=d\lambda P'(\lambda)$, we have
$$A_{d,N+1,d-1}=d\lambda P'(\lambda)A_{d,N,d}+A_{d,N,d-1}$$
so that $A_{d,N,d-1}=Nd \lambda P'(\lambda)$. Assume now that we have shown that $A_{d,N,r}\sim_{N\to+\infty}\omega_{d,r}N^{d-r}$ for $r=s+1,\dots,d$ and let us prove it for $s$. Rewriting \eqref{eq:inductionformula} we get
$$A_{d,N+1,s}=A_{d,N,s}+(s+1)\lambda P'(\lambda)A_{d,N,s+1}+\sum_{r=s+2}^d P(\lambda)^{r-s-1}A_{d,N,r}Q_{r,s}(\lambda).$$
We sum these equalities and use that $A_{d,N,r}\sim_{N\to+\infty}\omega_{d,r}N^{d-r}$ for $r\geq s+1$ to get
$$A_{d,N,s}=(s+1)\lambda P'(\lambda)\sum_{n=0}^N A_{d,n,s+1}+O\left(\sum_{n=0}^N n^{d-(s+2)}\right).$$
The result follows now easily.
\end{proof}

\subsection{Proof of Lemma \ref{lem:density}}
Let $u\in \ell^\infty(\NN)$ which is orthogonal to all $\left(\lambda^k\right)$ for $\lambda\in\Lambda$
and let $F(\lambda)=\langle u,\left(\lambda^k\right)\rangle$. Then $F$ is a holomorphic function in $\DD$ with an accumulation point of zeros inside $\DD$. Therefore, $F$ and $u$ are zero, which means that $\left\{\left(\lambda^k\right);\lambda\in\Lambda\right\}$ spans a dense subspace in $\ell^1(\NN)$.

\section{Convolution operators, the case $|\phi(0)|>1$}\label{sec:powers}

Example \ref{ex:philarge} relies clearly on the periodicity of the zeros of the sine function. We do not know what happens for other examples of entire functions with $|\phi(0)|>1$, for instance for $\phi(z)=e^z-\lambda$ with $|\lambda-1|>1$. Nevertheless, we have a general result for the existence of powers of hypercyclic vectors.

\begin{theorem}\label{thm:powersgeneral}
Let $X$ be an F-algebra and let $T\in\mathcal L(X)$. Assume that there exist a function $E:\CC\to X$ and
an entire function $\phi:\CC\to \CC$ satisfying the following assumptions:
\begin{enumerate}
\item for all $\lambda\in\CC$,\ $TE(\lambda)=\phi(\lambda)E(\lambda)$;
\item for all $\lambda,\mu\in\CC$, $E(\lambda)E(\mu)=E(\lambda+\mu)$;
\item for all $\Lambda\subset\CC$ with an accumulation point, $\textrm{span}\left(E(\lambda);\ \lambda\in\Lambda\right)$ is dense in $X$;
\item $\phi$ is not a multiple of an exponential function.
\end{enumerate}
Then there exists a residual set of vectors $u\in X$ such that, for all $m\geq 1$, $u^m\in HC(T)$.
\end{theorem}
\begin{proof}
Let us set $E=\left\{u\in X;\ \forall m\geq 1,\ u^m\in HC(T)\right\}$. Fixing $(V_j)_{j\geq 1}$ a basis of open subsets of $X$,
we set 
$$\mathcal O_{j,m}=\left\{u\in X;\ \exists N\geq 1,\ T^N (u^m)\in V_j\right\}.$$
Then $E=\bigcap_{j,m\geq 1}\mathcal O_{j,m}$ so that, since each $\mathcal O_{j,m}$ is clearly open, one just has to prove that these sets are dense. Thus, let $U,V$ be nonempty open subsets of $X$
and let $m\geq 1$. We are looking for a vector $u\in U$ and for an integer $N\geq 1$ such that $T^N(u^m)\in V$. 
Let $a\in\CC$ be such that $|\phi(a)|<1$. Arguing as in the proof of Corollary \ref{cor:smalleigengeneral}, we may find $w_0\in \CC$ with $|\phi(w_0)|>1$ and, for any $z\in\CC$ with $|z-a|\leq\frac{m-1}m |w_0-a|$, 
then $|\phi(z)|<1$. Then we fix $\delta>0$ sufficiently small so that $|\phi(z)|<1$ if $|z-a|\leq \frac{m-1}{m}|w_0-a|+\delta$ and $|\phi(w)|>1$ 
if $w\in B(w_0,\delta)$. Let finally, as usual(!), $w_1,w_2\in B(w_0,\delta)$ such that $t\in[0,1]\mapsto \log \left|\phi\left(tw_1+(1-t)w_2\right)\right|$ is strictly convex. 
One may find $p,q\in\NN$, complex numbers $a_1,\dots,a_p,b_1,\dots,b_q$, complex numbers $\gamma_1,\dots,\gamma_p\in B(a,\delta)$ and $\lambda_1,\dots,\lambda_q\in [w_1,w_2]$ such that
$$\sum_{l=1}^p a_l E\left(\gamma_l/m\right)\in U\textrm{ and }\sum_{j=1}^q b_j E\left(\lambda_j\right)\in V.$$
For $N\geq 1$ and $j\in\{1,\dots,q\}$, let $c_j:=c_j(N)$ be any complex number satisfying $c_j^m=b_j/\big(\phi(\lambda_j)\big)^{N}$ and
define
$$u:=u_N=\sum_{l=1}^p a_l E\left(\gamma_l/m\right)+\sum_{j=1}^q c_j E\left(\lambda_j/m\right)$$
so that
$$u^m=\sum_{d=0}^m \sum_{\substack{\bl\in I_p^{m-d} \\\bj\in I_q^d}}\alpha(\bl,\bj,d,m)a_\bl c_\bj E\left(\frac{\gamma_{l_1}+\cdots+\gamma_{l_{m-d}}+\lambda_{j_1}+\cdots+\lambda_{j_d}}m\right).$$
We write $\gamma_l=a+z_l$ with $|z_l|<\delta$ and $\lambda_j=w_0+z'_j$ with $|z'_j|<\delta$. Then
\[ \frac{\gamma_{l_1}+\cdots+\gamma_{l_{m-d}}+\lambda_{j_1}+\cdots+\lambda_{j_d}}m=\frac{(m-d)a+dw_0}{m}+Z \]
with $|Z|<\delta$. Moreover, provided $m<d$, 
\[ \left|\frac{(m-d)a+dw_0}m-a\right|=\frac dm \left|w_0-a\right|\leq\frac{m-1}m \left|w_0-a\right|. \]
Therefore, $\left|\phi\left(\frac{\gamma_{l_1}+\cdots+\gamma_{l_{m-d}}+\lambda_{j_1}+\cdots+\lambda_{j_d}}m\right)\right|<1$ and 
$$c_\bj(N)T^N E\left(\frac{\gamma_{l_1}+\cdots+\gamma_{l_{m-d}}+\lambda_{j_1}+\cdots+\lambda_{j_d}}m\right)\longrightarrow_{N\to+\infty}0.$$
When $d=m$, we conclude exactly as we have done before.
\end{proof}

\begin{proof}[Proof of Theorem \ref{thm:mainpowers}]
 That (i) or (ii) implies (iii) is already contained in \cite{ACPS07}. The proof of (iii) implies (i) is easy if we observe
 that, for any $\lambda\in\CC\backslash\{0\}$, a hypercyclic algebra for $\lambda\phi(D)$ is a supercyclic algebra for $\phi(D)$. 
 Finally, the implication (iii) implies (ii) is a consequence of Theorem \ref{thm:powersgeneral} for $T=\phi(D)$.
\end{proof}

\section{Infinitely generated and dense hypercyclic algebras} \label{sec:infinitelygenerated}

\subsection{Notations.} We use several specific notations for this section. We denote by $\NN$ the set of nonnegative integers
and by $\NNinf$ the set of sequences $(\alpha_1,\alpha_2,\dots)$ with $\alpha_i\in\NN$ for all $i$ and $\alpha_i=0$ for all large $i$. 
For $\alpha\in\NNinf$, $|\alpha|$ stands for the sum $\sum_{i=1}^{+\infty}\alpha_i$. If $A$ is a finite subset of $\NNinf$, $A\neq\varnothing$,
then $L(A)$ denotes $\sup\{|\alpha|;\ \alpha\in A\}$.

For $f\in X^\NN$ and $\alpha\in \NNinf$, with $\alpha_i=0$ for $i>d$, the notation $f^\alpha$ simply means the product $f_1^{\alpha_1}\cdots f_d^{\alpha_d}$. 
If $A$ is a finite subset of $\NNinfe$, we will often consider it as a subset of some $\NN^d$, since we may choose $d\geq 1$ such that $\alpha_i=0$
for all $i\geq d+1$ and all $\alpha\in A$. 

\subsection{A criterion à la Birkhoff for the existence of a dense and infinitely generated hypercyclic algebra}

To prove the existence of a dense and infinitely generated algebra of hypercyclic vectors, we need a reinforcement of Lemma \ref{lem:criterion}. 
This is achieved by the following natural proposition, which simplifies a statement of \cite{BP18} since it does not use the notion of pivot.

\begin{proposition}\label{prop:criterioninfinitely}
 Let $T$ be a continuous operator on the separable $F$-algebra $X$. Let $\prec$ be a total order on $\NN^{(\infty)}$. Assume that, for any $d\geq 1$, 
 for any finite and nonempty subset $A\subset\NNinfe$, for any nonempty open subsets $U_1,\dots,U_d,V$ of $X$, for any neighbourhood $W$ of $0$, 
 there exist $u=(u_1,\dots,u_d)\in U_1\times\cdots\times U_d$ and $N\geq 1$ such that, setting $\beta=\max(\alpha;\ \alpha\in A)$,
 \begin{align*}
  T^N(u^\beta)&\in V\\
  T^N(u^\alpha)&\in W\textrm{ for all }\alpha\in A,\ \alpha\neq\beta.
 \end{align*}
 Then $T$ admits a dense and  not finitely generated hypercyclic algebra.
\end{proposition}
\begin{proof}
Let $(V_k)$ be a basis of open neighbourhoods of $X$. For $A\subset\NNinfe$, $A\neq\varnothing$, $A$ finite, for $s,k\geq 1$, define
\begin{align*}
 E(A,s)&=\left\{P(z)=\sum_{\alpha\in A}\hat P(\alpha)z^\alpha;\ \hat P(\beta_A)=1\textrm{ and }|\hat P(\alpha)|\leq s\right\}\\
 \mathcal A(A,s,k)&=\left\{f\in X^\NN;\ \forall P\in E(A,s),\ \exists N\geq 1,\ T^N(P(f))\in V_k\right\}
\end{align*}
where $\beta_A=\max(\alpha;\ \alpha\in A)$. The set $A$ being fixed, $A$ may be considered as a subset of $\NN^d$ and $E(A,s)$ as a subset of $\CC[X_1,\dots,X_d]$. 
Moreover this set $E(A,s)$ is compact. By continuity of the maps $(f,P)\mapsto T^N(P(f))$, this implies that each set $\mathcal A(A,s,k)$ is open. Moreover, 
the assumptions of the theorem clearly imply that each such set is dense. Hence, $\mathcal G:=\bigcap_{A,s,k} \mathcal A(A,s,k)$ is a residual subset of $X^\NN$. 

Observe also that the set of $f$ in $X^\NN$ that induce a dense algebra in $X$ is residual in $X^\NN$ (see \cite{BP18}). Hence we may pick $f\in X^\NN$ belonging
to $\bigcap_{A,s,k}\mathcal A(A,s,k)$ and inducing a dense algebra in $X$. 

We show that for all nonzero polynomials $P$, $P(f)$ belongs to $HC(T)$. Let $A$ be the spectrum of $P$, let $\beta=\max(\alpha;\ \alpha\in A)$, let 
$Q=\frac{1}{\hat P(\beta)}P$ and let $s\geq 1$ be such that $Q\in E(A,s)$. Since $f$ belongs to $\bigcap_k \mathcal A(A,s,k)$, we conclude that $Q(f)$, hence
$P(f)$, is a hypercyclic vector for $T$.

It remains to show that the algebra generated by $f$ is not finitely generated. 
Assume on the contrary that it is generated by a finite number of $f^{\alpha(1)},\dots,f^{\alpha(p)}$. In particular, it is generated by a finite number
of $f_1,\dots,f_q$. Then there exists a polynomial $Q\in \CC[z_1,\dots,z_q]$ such that $f_{q+1}=Q(f_1,\dots,f_q)$. Define $P(z)=z_{q+1}-Q(z)$. Then 
$P$ is a nonzero polynomial. Nevertheless, $P(f)=0$, which contradicts the fact that $P(f)$ is a hypercyclic vector for $T$.
\end{proof}

\begin{remark}
 The last part of the proof may be formulated in the following way: let $f\in X^\NN$ be such that $P(f)$ is never zero for any nonzero polynomial $P$. 
 Then the algebra generated by $f$ is not finitely generated. We could avoid this by using the following lemma, proved in \cite{BP18}: the set of
 sequences $f\in X^\NN$ whose induced algebra is not finitely generated is residual in $X$. Nevertheless, the proof of this last statement seems more complicated.
\end{remark}

\subsection{On the existence of infinitely generated hypercyclic algebras}

We now show how to adapt our proofs to the existence of a dense and infinitely generated algebra. We have to pay the price of additional technical difficulties
and we restrict ourselves to an analogue of Corollary  \ref{cor:smalleigengeneralbis}.

\begin{theorem}\label{thm:smalleigeninfinitely}
Let $X$ be an F-algebra and let $T\in\mathcal L(X)$. Assume that there exist a function $E:\CC\to X$ and
an entire function $\phi:\CC\to \CC$ satisfying the following assumptions:
\begin{enumerate}
\item for all $\lambda\in\CC$,\ $TE(\lambda)=\phi(\lambda)E(\lambda)$;
\item for all $\lambda,\mu\in\CC$, $E(\lambda)E(\mu)=E(\lambda+\mu)$;
\item for all $\Lambda\subset\CC$ with an accumulation point, $\textrm{span}\left(E(\lambda);\ \lambda\in\Lambda\right)$ is dense in $X$;
\item $\phi$ is not a multiple of an exponential function;
\item for all $\rho\in (0,1)$, there exists $w_0\in\CC$ with $\left|\phi\left(w_0\right)\right|>1$ and, for all $r\in (0,\rho]$, $\left|\phi\left(rw_0\right)\right|<1$.
\end{enumerate}
Then $T$ supports a hypercyclic algebra which is dense and is not finitely generated.
\end{theorem}
\begin{proof}
We intend to apply Proposition \ref{prop:criterioninfinitely}. Thus,
 let $d\geq 1$, let $A$ be a finite and nonempty subset of $\NNinfe$. Enlarging $d$ if necessary, we may and shall assume that $A\subset\NN^d$. 
 We choose for total order on $\NN^d$ the lexicographical order and we denote by $\beta$ the maximal element of $A$ for this order. Without loss
 of generality, we assume that $\beta_1\neq 0$. Let also $U_1,\dots,U_d,V$ be nonempty open subsets of $X$ and let $W$ be an open neighbourhood of $0$.
 We set
\begin{align*}
 I_\beta&=\left\{i\in\{2,\dots,d\};\ \beta_i\neq 0\right\}\\
 \Omega_A&=\left\{\alpha\in A;\ \alpha_1=\beta_1\right\}\backslash\{\beta\}\\
 &\quad\quad\cup\left\{\alpha\in \NN^d;\ \forall i\in I_\beta,\ \alpha_i\leq\beta_i\textrm{ and }\exists i\in I_\beta,\ \alpha_i<\beta_i\right\}.
\end{align*}
We first consider the case $I_\beta\neq \varnothing$ (the other case, which is easier, will be discussed at the end of the proof). 
Observe that, for any $\alpha\in \Omega_A$, there exists $i_0\in I_\beta$ such that, for all $i\leq i_0$, $\alpha_i=\beta_i$ and 
$\alpha_{i_0}<\beta_{i_0}$. 
Therefore it is easy to construct a sequence $(\rho_i)_{i\in I_\beta}\subset (0,1)$ satisfying
$$\sum_{i\in I_\beta}\rho_i=1$$
$$\forall\alpha\in \Omega_A,\ \sum_{i\in I_\beta}\rho_i\times\frac{\alpha_i}{\beta_i}<1.$$
 Let $\eta>0$ be such that
$$\forall \alpha\in\Omega_A,\ \sum_{i\in I_\beta}\rho_i\times\frac{\alpha_i}{\beta_i}\leq 1-\eta.$$
We finally choose $\veps>0$ and $\rho\in (0,1)$ satisfying
$$\left\{
\begin{array}{rcl}
 \rho&>&\displaystyle(1-\veps)\frac{\beta_1-1}{\beta_1}+L(A)\veps\\
 \rho&>&(1-\veps)+(1-\eta)\veps=1-\eta\veps.
\end{array}
\right.$$
This is possible for instance by setting $\rho=1-\eta\veps/2$ for a sufficiently small $\veps>0$. For this value of $\rho$, we get $w_0\in\CC$
with $|\phi(w_0)|>1$ and $|\phi(rw_0)|<1$ if $r\in (0,\rho]$. We set $\kappa=\veps w_0$ and $z_0=(1-\veps)w_0=w_0-\kappa$ and we summarize some properties of $w_0$, $z_0$ and $\kappa$ below:
\begin{equation}\label{eq:infinitely1}
 |\phi(w_0)|>1
\end{equation}
\begin{equation}\label{eq:infinitely2}
 |\phi(z_0+r\kappa)|<1\textrm{ if }r\in [0,1-\eta]
\end{equation}
\begin{equation}\label{eq:infinitely3}
 |\phi(tz_0+s\kappa)|<1\textrm{ if }t\leq\frac{\beta_1-1}{\beta_1}\textrm{ and }s\leq L(A).
\end{equation}
By continuity, there exists $\delta>0$ such that these properties remain valid respectively in $B(w_0,\delta)$, $B(z_0+r\kappa,\delta)$, $B(tz_0+s\kappa,\delta)$.
Let (as usual!) $w_1\neq w_2$ in $B(w_0,\delta/2)$ such that $t\mapsto \log|\phi(tw_1+(1-t)w_2)|$ is stricly convex on $[0,1]$. 
We then choose integers $p,q$, complex numbers $a_{1,i},\dots,a_{p_i,i}$, $b_1,\dots,b_q$, complex numbers $\gamma_{1,i},\dots,\gamma_{p_i,i}\in B(0,\delta/2L(A))\cap (0,\rho w_0)$,
complex numbers $\lambda_1,\dots,\lambda_q$ in $[w_1,w_2]$ such that
$$\forall i\in\{1,\dots,d\},\ \sum_{l=1}^{p_i}a_{l,i}E(\gamma_{l,i})\in U_i\textrm{ and }\sum_{j=1}^q b_j E(\lambda_j)\in V.$$
We define for $j=1,\dots, q$, $z_j=\lambda_j-\kappa$. We then set
$$u_1(N)=\sum_{l=1}^{p_1} a_{l,1}E(\gamma_{l,1})+\sum_{j=1}^q c_j(N)E(z_j/\beta_1),$$
$$u_i=\sum_{l=1}^{p_i}a_{l,i}E(\gamma_{l,i})+\omega E(\rho_i \kappa/\beta_i),$$
provided $i\in I_\beta$,
$$u_i=\sum_{l=1}^{p_i}a_{l,i}E(\gamma_{l,i})$$
otherwise, namely if $i\notin I_\beta$ and $i>1$. Above, $\omega$ is any positive real number small enough so that all $u_i$ belong to $U_i$ for $i\geq 2$, and $c_j(N)$ is any complex number such that
$$c_j(N)^{\beta_1}=\frac{b_j}{\big(\phi(\lambda_j)\big)^N\omega^{\sum_{i\in I_\beta}\beta_i}}.$$
Since $c_j(N)$ goes to zero, $u_1(N)$ belongs to $U_1$ for $N$ large enough. It remains to show that $T^N(u^\beta)\in V$ and $T^N(u^\alpha)\in W$
for $\alpha\in A\backslash\{\beta\}$ and $N$ large enough. Let us first compute $u^\beta$. Let us examine $u_1^{\beta_1}$. As in the proof of 
Theorem \ref{thm:philarge}, we can distinguish three different kinds of terms:
\begin{itemize}
 \item the terms $c_j(N)^{\beta_1}E(z_j)$, for $j=1,\dots,q$;
 \item the terms $c_\bj(N)E\left(\frac{z_{j_1}+\cdots+z_{j_{\beta_1}}}{\beta_1}\right)$ for $\bj\in I_q^{\beta_1}$ nondiagonal;
 \item a finite number of terms $a(\bj,\gamma,N)E\left(\frac{z_{j_1}+\cdots+z_{j_l}}{\beta_1}+\gamma\right)$ with $\bj\in I_q^l$, $l<\beta_1$
 and $|\gamma|<\beta_1 \delta/2L(A)$.
\end{itemize}
Observe that the terms $a(\bj,\gamma,N)$ may depend on $N$ by involving $c_{\bj}(N)$ (this happens if $\bj\in I_q^l$ for $l>0$) but that they are uniformly bounded in $N$.

Let us now inspect $u_2^{\beta_2}\cdots u_d^{\beta_d}$. For this product, we distinguish two kinds of terms:
\begin{itemize}
 \item the term $\omega^{\sum_{i\in I_\beta}\beta_i}E(\kappa)$;
 \item a finite number of terms $b(r,\gamma')E(r\kappa+\gamma')$ with $0\leq r\leq 1-\eta$ and $|\gamma'|< (\beta_2+\cdots+\beta_d)\delta/2L(A)\leq\delta/2$.
 Here, $r$ is equal to $\sum_{i\in I_\beta}\alpha_i\rho_i/\beta_i$ for some $\alpha\in\Omega_A$, which explains why $r\leq 1-\eta$. 
\end{itemize}
Thus, taking the product $u^{\beta}$, we get six kinds of terms:
\begin{itemize}
\item the terms $c_j(N)^{\beta_1}\omega^{\sum_{i\in I_\beta}\beta_i}E(\lambda_j)$ for $j=1,\dots,q$. But the choice of $c_j(N)$ ensures that 
$$T^N\left(\sum_{j=1}^q c_j(N)^{\beta_1}\omega^{\sum_{i\in I_\beta}\beta_i}E(\lambda_j)\right)=\sum_{j=1}^q b_jE(\lambda_j).$$
\item the terms $c_j(N)^{\beta_1}b(r,\gamma')E(z_j+r\kappa+\gamma')$, with $1\leq j\leq q$, $0\leq r\leq 1-\eta$ and $|\gamma'|<\delta/2$. Since $|z_j+r\kappa+\gamma'-(z_0+r\kappa)|<\delta$, we deduce from \eqref{eq:infinitely2} that
$$T^N\left(c_j(N)^{\beta_1}b(r,\gamma')E(z_j+r\kappa+\gamma')\right)\rightarrow_{N\to+\infty}0.$$
\item the terms $c_\bj(N)\omega^{\sum_{i\in I_\beta}\beta_i}E\left(\frac{\lambda_{j_1}+\cdots+{\lambda_{j_{\beta_1}}}}{\beta_1}\right)$ for $\bj\in I_q^{\beta_1}$ nondiagonal. But as in the proof of Theorem \ref{thm:philarge}, the choice of $[w_1,w_2]$ ensures that for all $\bj\in I_q^{\beta_1}$ nondiagonal, 
$$T^N\left(c_\bj(N)\omega^{\sum_{i\in I_\beta}\beta_i}E\left(\frac{\lambda_{j_1}+\cdots+{\lambda_{j_{\beta_1}}}}{\beta_1}\right)\right)\rightarrow_{N\to+\infty}0.$$
\item the terms $c_\bj(N)b(r,\gamma')E\left(\frac{z_{j_1}+\cdots+{z_{j_{\beta_1}}}}{\beta_1}
+r\kappa+\gamma'\right)$. As before, \eqref{eq:infinitely2} ensures that 
$$T^N\left(c_\bj(N)b(r,\gamma')E\left(\frac{z_{j_1}+\cdots+{z_{j_{\beta_1}}}}{\beta_1}
+r\kappa+\gamma'\right)\right)\rightarrow_{N\to+\infty}0.$$
\item the terms $a(\bj,\gamma,N)\omega^{\sum_{i\in I_\beta}\beta_i}E\left(\frac{z_{j_1}+\cdots+z_{j_l}}{\beta_1}+\kappa+\gamma\right)$ with $l<\beta_1$, $|\gamma|<\beta_1 \delta/2L(A)$. But writing 
$$\frac{z_{j_1}+\cdots+z_{j_l}}{\beta_1}+\kappa=\frac{l}{\beta_1}z_0+\kappa+\gamma'$$
with $|\gamma'|<\delta$, we deduce from \eqref{eq:infinitely3} that 
$$T^N\left(a(\bj,\gamma,N)\omega^{\sum_{i\in I_\beta}\beta_i}E\left(\frac{z_{j_1}+\cdots+z_{j_l}}{\beta_1}+\kappa+\gamma\right)\right)\rightarrow_{N\to+\infty}0.$$
\item the terms $a(\bj,\gamma,N)b(r,\gamma')E\left(\frac{z_{j_1}+\cdots+z_{j_l}}{\beta_1}+\gamma+r\kappa+\gamma'\right)$ with $l<\beta_1$, $|\gamma|+|\gamma'|<\delta/2$ and $r\in [0,1-\eta]$. A similar proof shows that 
$$T^N\left(a(\bj,\gamma,N)b(r,\gamma')E\left(\frac{z_{j_1}+\cdots+z_{j_l}}{\beta_1}+\gamma+r\kappa+\gamma'\right)\right)\rightarrow_{N\to+\infty}0.$$
\end{itemize}
Hence, as expected, provided $N$ is large enough, $T^N(u(N)^\beta)\in V$. Let us now consider $\alpha\in A$ with $\alpha\prec \beta$ and let us show that
$T^N(u(N)^\alpha)$ goes to zero as $N$ goes to $+\infty$. The analysis is similar but simpler. Either $\alpha_1=\beta_1$ and in that case in $u_2^{\alpha_2}\cdots u_d^{\alpha_d}$ 
appear only terms like $b(r,\gamma)E(r\kappa+\gamma)$ with $0\leq r\leq 1-\eta$ and $|\gamma|<(\alpha_2+\cdots+\alpha_d)\delta/2L(A)$. 
Here, $r=\sum_{i\in I_\beta}\alpha'_i\rho_i/\beta_i$ for some $\alpha'_i\leq \alpha_i$ so that $r\leq 1-\eta$ since $\alpha\in \Omega_A$. 
Then applying $T^N$ to each term of $u^\alpha$ will lead to a sequence going to zero. 

Or $\alpha_1<\beta_1$, and now in $u_1^{\alpha_1}$ appear only terms like 
$$a(\bj,\bl,N)E\left(\frac{z_{j_1}+\cdots+z_{j_t}}{\beta_1}+\gamma_{l_1,1}+\cdots+\gamma_{l_{\alpha_1-t},1}\right)$$
with $t\leq \alpha_1$, namely terms like $a(\bj,\gamma,N)E\left(\frac{t}{\beta_1}z_0+\gamma\right)$ with $|\gamma|<\delta/2$.
Expanding the product $u_2^{\alpha_2}\cdots u_d^{\alpha_d}$ leads to terms like $b(r,\gamma')E(r\kappa+\gamma')$ with $|\gamma'|<\delta/2$ and 
$$r\leq \sum_{i\in I_\beta}\rho_i\frac{\alpha_i}{\beta_i}\leq L(A).$$
A last application of \eqref{eq:infinitely3} shows that 
$$T^N\left(a(\bl,\gamma,N)b(r,\gamma')E\left(\frac t{\beta_1}z_0+r\kappa+\gamma+\gamma'\right)\right)\rightarrow_{N\to+\infty}0.$$
We need finally to consider the case $I_\beta=\varnothing$. In that case, the proof of Theorem \ref{thm:philarge} works almost mutatis mutandis. Indeed, the situation is simplified
because now $u^\beta=u_1^{\beta_1}$. We then set
\begin{align*}
u_1(N)&=\sum_{l=1}^{p_1}a_{l,1}E(\gamma_{l,1})+\sum_{j=1}^q c_j(N)E\left(\frac{\lambda_j}{\beta_1}\right)\\
u_i&=\sum_{l=1}^{p_i}a_{l,i}E(\gamma_{l,i}),\ i=2,\dots,d
\end{align*}
with $c_j(N)^{\beta_1}=b_j/\left(\phi(\lambda_j)\right)^N$ and we follow a completely similar proof. Observe in particular that if $\alpha\in A$ satisfies $\alpha\prec\beta$, then $\alpha_1<\beta_1$.
\end{proof}

%
%

\bibliographystyle{plain} 
\bibliography{biblio,bibliomoi,preprint} 

\end{document}